\documentclass{amsart}
\usepackage{pgf,pgfarrows,pgfnodes,pgfautomata,pgfheaps,pgfshade}
\usepackage{
graphicx, amsfonts, amssymb, amsbsy, amsmath, amsthm, url}

\theoremstyle{plain}
\newtheorem{theorem}{Theorem}[section]

\newtheorem{corollary}[theorem]{Corollary}
\newtheorem{proposition}[theorem]{Proposition}

\theoremstyle{definition}
\newtheorem{definition}[theorem]{Definition}

\theoremstyle{remark}
\newtheorem{remark}{Remark}

\newcommand{\field}[1]{\mathbb{#1}}

\newcommand{\Z}{\field{Z}}
\newcommand{\design}{{\mathcal D}}
\newcommand{\fraction}{{\mathcal F}}

\begin{document}

\title[On the aberrations of Orthogonal Arrays with removed runs]{On the aberrations of mixed level Orthogonal Arrays with removed runs}

\author{Roberto Fontana}
\address{Department DISMA, Politecnico di Torino, Torino, Italy}
\email{roberto.fontana@polito.it}
\author{Fabio Rapallo}
\address{Department DISIT, University of Piemonte Orientale, Alessandria, Italy}
\email{fabio.rapallo@uniupo.it}

\begin{abstract}
Given an Orthogonal Array we analyze the aberrations of the sub-fractions which are obtained by the deletion of some of its points. We provide formulae to compute the Generalized Word-Length Pattern of any sub-fraction. In the case of the deletion of one single point, we provide a simple methodology to find which the best sub-fractions are according to the Generalized Minimum Aberration criterion. We also study the effect of the deletion of 1, 2 or 3 points on some examples. The methodology does not put any restriction on the number of levels of each factor. It follows that any mixed level Orthogonal Array can be considered.
\end{abstract}



\maketitle

\section{Introduction}
\label{sec:1}

The theory of Orthogonal Arrays (OAs) has a long history which began with Combinatorics and is today a major research topic in both
methodological and applied Statistics. For instance, OAs are used in industrial experimentation to determine the optimum mix of factors to predict a response variable. The need for efficient experimental designs has led to the definition of several criteria for the choice of the design points. All such criteria aim to produce the best estimates of the relevant parameters for a given sample size. As general references for OAs, the reader can refer to \cite{hedayat2012orthogonal}, \cite{dey2009fractional} and \cite{mukerjee2007modern}.

An important object associated to a design is its Generalized Word-Length Pattern (GWLP). The GWLP is used to discriminate among different designs through the Generalized Minimum Aberration (GMA) criterion: given two designs ${\mathcal F}_1$ and ${\mathcal F}_2$ with $m$ factors, the corresponding GWLPs are two vectors
\[
A_{\fraction_i}=\left(A_0(\fraction_i)=1, A_1(\fraction_i), \ldots , A_m(\fraction_i)\right) \qquad i=1,2 \, .
\]
The GMA criterion is defined as the sequential minimization of these GWLPs: $\fraction_1$ is better than $\fraction_2$ if there exists $j$ such that $A_0(\fraction_1) = A_0(\fraction_2), \ldots, A_j(\fraction_1) = A_j(\fraction_2)$ and $A_{j+1}(\fraction_1) < A_{j+1}(\fraction_2)$. The GMA criterion is usually applied to Orthogonal Arrays (OA), see \cite{hedayat2012orthogonal}. The GMA criterion was introduced in \cite{fries|hunter:80} for two-level designs and then extended to non-regular multilevel designs in \cite{xu:wu}.

While for binary designs the meaning of the GWLP is simple, since it counts the number of defining equations with a given length, in the multilevel case the interpretation of the GWLP is in general less intuitive. Nevertheless, the GWLP is a measure of the aliasing among the simple terms and interaction terms of the design, see e.g. \cite{gromping|xu:14}.

In this paper, we consider fractional factorial designs under the GMA criterion, and we focus on the order of the runs, motivated by the following practical problem. In several situations, it is hard to fix the sample size {\it a priori}. For example, budget constraints or time limitations may occur after the definition of the design, or even when the experiments are running, thus leading to an incomplete design. In such a situation, it is relevant not only to choose an OA with good properties, but also to define an order of the design points, so that the experimenter can stop the sequence of runs and loose as little information as possible. While OAs with added runs are well studied, see for instance \cite{Chatzopoulos2011}, less has been done in the case of OAs with removed runs. Some results in this direction are introduced in \cite{butler07} and in \cite{street}, but in those papers the goal of the analysis was to maximize the $D$-optimality, and therefore those works consider a model-based environment. Some examples of two-level OAs with removed runs is discussed in \cite{xampeny}. A different approach to the problem of the order of the runs is presented in \cite{thestatistician} for the two-level case. Although the goal of the analysis is different Wang and Jan discuss several practical problems involving OAs with missing runs and include some useful references.

In this work, we consider the GMA criterion for binary OAs and we study the behavior of the aberrations when $p$ points are removed from an OA, for small values of $p$. Since the GWLP does not depend on the coding of the factor levels, in \cite{pistone|rogantin:08} the use of the complex coding is suggested to express the basis of the polynomial complex functions over a design, and therefore the counting function. The choice of the complex coding is particularly useful for simplifying the expressions of the aberrations and of the GWLP. Moreover, the coefficients of the counting function can be expressed in terms of the counts of the levels appearing in each simple or interaction term.

In particular, we write the GWLP of a fraction through a row-wise decomposition, showing that the GWLP of a fraction depends only on the mutual position of the design points inside the grid of the corresponding full factorial design. This approach is alternative to the classical decomposition of the GWLP in terms of the aberrations and allows us to easily compute the GWLP for sub-fractions, i.e. fractions with removed runs.

The paper is structured as follows. After a brief summary of the basic definitions concerning counting functions, aberrations, GWLP, and OAs (Sect.~\ref{sec:alg}), we introduce some formulae to decompose the GWLP of a fraction (Sect.~\ref{sec:formulae_GWLP}) and we use such formulae to analyze several examples of OAs with removed runs (Sect.~\ref{sec:examples}). Moreover, a recursive formula for the two-level case is introduced to further simplify the computations (Sect.~\ref{sec:twolev}). Finally, some pointers to future works are briefly discussed (Sect.~\ref{sec:finalrem}).

\section{Fractions, counting functions and aberration}
\label{sec:alg}

In this section we collect some relevant definitions and results on fractions of factorial designs and their representation through polynomial counting functions to fix the notation and to make explicit some major formulae to be used later. In our presentation we use the complex coding of the factor levels, as done in, e.g., \cite{pistone|rogantin:08}, although most of the results remains valid for any choice of an orthonormal choice of the coding. Moreover, we briefly summarize some properties of OAs in order to highlight the connections between orthogonality and the coefficients of the counting function. In particular, we recall the results expressing the aberrations and the GWLP of a fractional design in terms of the coefficients of the polynomial counting function.

Let us consider an experiment with $m$ factors $\design_{j}, \; j=1,\ldots,m$. Let us code the $s_j$ levels of the factor $\design_{j}$ by  the $s_j$-th roots of the unity
\[
\design_{j} = \{\omega_0^{(s_j)},\ldots,\omega_{s_j-1}^{(s_j)}\} \, ,
\]
where $\omega_k^{(s_j)}=\exp\left(\sqrt{-1}\:  \frac {2\pi}{s_j} \ k\right)$, $k=0,\ldots,s_j-1, \ j=1,\ldots,m$.

The \emph{full factorial design} with complex coding is $\design = \design_1 \times \cdots \times \design_m$. We denote its cardinality by $\# \design=\prod_{j=1}^m s_j$. A \emph{fraction} $\fraction$ is a multiset $(\fraction_*,f_*)$ whose underlying set of elements $\fraction_*$ is a subset of $\design$ and $f_*$ is the multiplicity function $f_*: \fraction_* \rightarrow \mathbb N$ that gives the number of replicates of each design point of $\fraction_*$ in the multiset $\fraction$. When a fraction is a single-replicate fraction, i.e., when $f^*(\zeta)= 1$ for all $\zeta \in \fraction_*$ we will identify $\fraction$ and $\fraction_*$ with a slight abuse of notation. The number of design points in $\fraction$ is denoted with $n = \# \fraction= \sum_{\zeta \in \fraction_*} f_*(\zeta)$.

A basis of the vector space of the complex functions over $\design$ can be defined as follows. Define the \emph{exponent set}:
\[
L =  \Z_{s_1} \times \cdots \times  \Z_{s_m}
\]
and the projections onto single factors:
\[
X_j \colon \design \ni (\zeta_1,\ldots,\zeta_m)\ \longmapsto \ \zeta_j \in \design_j \, .
\]
Then, a basis of the complex functions over $\design$ is formed by
\[
\{X^\alpha=X_1^{\alpha_1} \cdot \ldots \cdot X_m^{\alpha_m}, \qquad  \alpha \in L \} \, .
\]
Let $| \alpha |_0$ be the number of non-null elements of $\alpha$. Following the standard terminology in factorial design theory, the monomials ${X^\alpha}$ are called \emph{factors} when $| \alpha |_0=1$ and \emph{interaction terms} when $| \alpha |_0 \geq 2$. We use this basis to represent the counting function of a fraction according to the following definition.

\begin{definition} \label{de:indicator}
The \emph{counting function} $R$ of a fraction $\fraction$ is a complex polynomial defined over $\design$ so that for each $\zeta \in \design$, $R(\zeta)$ equals the number of appearances of $\zeta$ in the fraction. A $0-1$ valued counting function is called an \emph{indicator function} of a single-replicate fraction $\fraction$. We denote by $c_\alpha$ the coefficients of the representation of $R$  on $\design$ using the monomial basis
$\{X^\alpha, \ \alpha \in L\}$:
\[
R(\zeta) = \sum_{\alpha \in L} c_\alpha X^\alpha(\zeta), \;\zeta\in\design, \;  c_\alpha \in \mathbb C \, .
\]
\end{definition}

%

In the proposition below we summarize some properties of the coefficients of the counting function. The proof of all items can be found in \cite{pistone|rogantin:08}. Here $\overline{z}$ is the complex conjugate of the complex number $z$.

\begin{proposition} \label{pr:bc-alpha}
If $\fraction$ is a fraction of a full factorial design $\design$, $R = \sum_{\alpha \in L} c_\alpha X^\alpha$ is its counting function and $[\alpha-\beta]$ is the $m$-tuple made by the componentwise difference in the rings $\Z_{s_j}$,
 $\left(\left[\alpha_1-\beta_1 \right]_{s_1}, \ldots, \left[\alpha_m - \beta_m\right]_{s_m} \right)$, then
\begin{enumerate}
 \item \label{it:balpha}
the coefficients $c_\alpha$ are given by $c_\alpha= \frac 1 {\#\design} \sum_{\zeta \in \fraction} \overline{X^\alpha(\zeta)}$;
 \item \label{it:cent1}
 the term $X^\alpha$ is centered on $\fraction$ i.e., $\frac{1}{\#\fraction} \sum_{\zeta \in \fraction} X^\alpha(\zeta)=0$ if, and only if,
 $c_\alpha=c_{[-\alpha]}=0 $;
 \item \label{it:cent2}
 the terms $X^\alpha$ and $X^\beta$
are orthogonal on $\fraction$ if and only if,
$c_{[\alpha-\beta]}=0 $.
\end{enumerate}
\end{proposition}

Proposition \ref{pr:bc-alpha} has a major application in the representation of the Orthogonal Arrays through the counting function. Recall that, given a subset of indices $I=\{i_1,\ldots,i_k\} \subset \{1,\ldots,m\}, i_1<\ldots < i_k$, the projection of a design point $\zeta$ into $\design_I:=\design_{i_1} \times \cdots \times \design_{i_k}$ is
\[
\pi_I(\zeta)=\zeta_I \equiv (\zeta_{i_1},\ldots,\zeta_{i_k}) \in \design_{i_1} \times \ldots \times \design_{i_k} \, .
\]
A fraction $\fraction$ {\em factorially projects} onto the $I$-factors, $I=\{i_1,\ldots,i_k\} \subset \{1,\ldots,m\}$, $i_1<\ldots < i_k$, if the projection $\pi_I(\fraction)$ is a multiple of a full factorial design, i.e., the multiset $(\design_{i_1} \times \ldots \times \design_{i_k} , f_*)$ where the multiplicity function $f_*$ is constant over $\design_{i_1} \times \ldots \times \design_{i_k}$.

\begin{definition}
A fraction $\fraction$ is a {\em (mixed level) Orthogonal Array (OA)} of strength $t$ if it factorially projects onto any $I$-factors with $\#I=t$.
\end{definition}



From Proposition \ref{pr:bc-alpha} it follows that a fraction factorially projects onto the $I$-factors, $I=\{i_1,\ldots,i_k\} \subset \{1,\ldots,m\}, i_1<\ldots < i_k$, if and only if all the coefficients of the counting function involving the $I$-factors only are $0$. Thus, a fraction is an OA of strength $t$ if and only if all the coefficients $c_{\alpha}, \; \alpha \neq 0 \equiv (0,\ldots,0)$ of the counting function up to the order $t$ are $0$.

Using the polynomial counting function, \cite{cheng2004geometric} provides the following definition of the GWLP $A_\fraction=(A_0(\fraction), \ldots, A_m(\fraction))$ of a fraction $\fraction$.

\begin{definition} \label{gwlp}
The Generalized Word-Length Pattern (GWLP) of a fraction $\fraction$ of the full factorial design $\design$ is a the vector $A_\fraction=(A_0(\fraction),A_1(\fraction), \ldots , A_m(\fraction))$, where
\[
A_j(\fraction)= \sum_{|\alpha |_0 =j} a_\alpha \quad j=0,\ldots,m \, ,
\]
\begin{equation} \label{aberration}
a_\alpha = \left(  \frac{ \|c_{\alpha}\|_2 }{c_{0}} \right)^2 \, ,
\end{equation}
$\| z \|_2$ is the norm of the complex number $z$, and $c_0 := c_{(0,\ldots,0)}=n/{\#\design}$.
\end{definition}

We refer to $a_\alpha$ as the \emph{aberration} of the interaction $X^\alpha$. Note that $A_0(\fraction)=1$ for all $\fraction$.

\section{Decomposition of the GWLP and applications} \label{sec:formulae_GWLP}

In this section we show some formulae to compute the GWLP of a fraction. The first proposition shows that the elements of the GWLP of a fraction depend only on the mutual position of its design points.

\begin{proposition} \label{pr:gwlp_mat}
Given a fraction $\fraction$ of size $n$,
\begin{equation} \label{formula_S}
n^2 A_j(\fraction)=\sum_{f\in\fraction} \sum_{g\in\fraction} \sum_{\substack{A=\{a_1,\ldots,a_j\} \subseteq \{1,\ldots,m\}}} S_{a_1}^{(f,g)}\cdot \ldots \cdot S_{a_j}^{(f,g)}
\end{equation}
where
\[
S_{i}^{(f,g)}=
\begin{cases}
-1 & \text{ if } f_i \neq g_i \\
s_i-1 & \text{ if } f_i = g_i
\end{cases}
\; i=1,\ldots,m
\]
and $j=1,\ldots,m$.
\end{proposition}
\begin{proof}
Given a fraction $\fraction$, let $R(\zeta) = \sum_{\alpha \in L} c_\alpha X^\alpha(\zeta), \;\zeta\in\design, \;  c_\alpha \in \mathbb C$ be its counting function. From Eq.~\eqref{aberration}, the $j$-th term of the GWLP of $\fraction$ is
\[
A_j(\fraction)= \sum_{|\alpha |_0 =j} \left(  \frac{ \|c_{\alpha}\|_2 }{c_{0}} \right)^2 \, \quad j=0,\ldots,m \, .
\]
The counting function $R$ can be written as the sum of the counting functions of the points of $\fraction$. Let $R^{(f)}$ be the indicator function of a point $f \in \design$. We have
\begin{equation}
R=\sum_{f \in \fraction}R^{(f)}.
\label{eq:sum_R}
\end{equation}
We observe that $\fraction$ can be a multiset, i.e. it can exist $\zeta_\star \in \fraction$ such that $R(\zeta_\star)>1$. The corresponding term in Eq.~\eqref{eq:sum_R} will be
\[
\underbrace{R^{(\zeta_\star)}+\ldots+R^{(\zeta_\star)}}_{\text{$R(\zeta_\star)$ times}}.
\]
We denote by $c_\alpha^{(f)}$ the coefficients of the indicator function of the point $f$, $R^{(f)}=c_\alpha^{(f)}X^\alpha$.
Let us consider the case $m=1$, $\design=\{\omega_0^{(s_1)},\ldots,\omega_{s_1-1}^{(s_1)}\}$. It follows that $\design=\{\omega_f^{(s_1)}: f=0,1,\ldots, s_1-1\}$ and a generic point $f \in \design$ can be written, with a small abuse of notation, as $\omega_f^{(s_1)}$.
It is not difficult to show that
\[
R^{(f)}(\zeta)=\frac{1}{s_1}(1+f^{s_1-1}\zeta+f^{s_1-2}\zeta^2+\ldots+f\zeta^{(s_1-1)})
\]
that is
\[
c_\alpha^{(f)}=\frac{1}{s_1}f^{s_1-\alpha}=\frac{1}{s_1}(\omega_f^{(s_1)})^{s_1-\alpha}=\frac{\omega_{-\alpha f}^{(s_1)}}{s_1} \; \;  \text{ for }\alpha=0,1,\ldots, s_1-1.
\]
The generalization to the case $m>1$ is straightforward. Given $f=(f_1,\ldots,f_m)\equiv(\omega_{f_1}^{(s_1)},\ldots,\omega_{f_m}^{(s_m)})\in \design$ we get
\[
c_\alpha^{(f)}= \frac{1}{\#\design} \omega_{-\alpha_1 f_1}^{(s_1)}\cdot \ldots \cdot \omega_{-\alpha_m f_m}^{(s_m)},\; \; \; \; \alpha \in L.
\]
We can write
\[
c_{0}^2 A_j(\fraction)= \sum_{|\alpha |_0 =j}  \|c_{\alpha}\|_2 ^2 = \sum_{f \in \fraction}\sum_{g \in \fraction}\sum_{|\alpha |_0 =j} c_\alpha^{(f)} \overline{c}_\alpha^{(g)},
\, \quad j=0,\ldots,m
\]
and we obtain
\[
c_\alpha^{(f)} \overline{c}_\alpha^{(g)}=\frac{1}{\#\design^2} \omega_{\alpha_1(g_1-f_1)}^{(s_1)}\cdot \ldots \cdot \omega_{\alpha_m(g_m-f_m)}^{(s_m)} .
\]
Let us consider the $m$-th term of the GWLP of $\fraction$. We get
\begin{eqnarray}
c_{0}^2 A_m(\fraction)= \sum_{f \in \fraction}\sum_{g \in \fraction}\sum_{|\alpha |_0 =m} c_\alpha^{(f)} \overline{c}_\alpha^{(g)}= \\
= \sum_{f \in \fraction}\sum_{g \in \fraction} \sum_{\alpha_1=1}^{s_1-1} \ldots \sum_{\alpha_m=1}^{s_m-1}  \frac{1}{\#\design^2} \omega_{\alpha_1(g_1-f_1)}^{(s_1)}\cdot \ldots \cdot \omega_{\alpha_m(g_m-f_m)}^{(s_m)} = \\
= \frac{1}{\#\design^2} \sum_{f \in \fraction}\sum_{g \in \fraction} \sum_{\alpha_1=1}^{s_1-1} \omega_{\alpha_1(g_1-f_1)}^{(s_1)} \ldots \sum_{\alpha_m=1}^{s_m-1} \omega_{\alpha_m(g_m-f_m)}^{(s_m)}
\label{eq:am}
\end{eqnarray}
We observe that, for $i=1, \ldots, m$ we have
\[
\sum_{\alpha_i=1}^{s_i-1} \omega_{\alpha_i(g_i-f_i)}^{(s_i)}=
\begin{cases}
\sum_{\alpha_i=0}^{s_i-1} \omega_{\alpha_i(g_i-f_i)}^{(s_i)}-\omega_0^{(s_i)}=-1 & \text{ if } f_i \neq g_i \\
\sum_{\alpha_i=1}^{s_i-1} \omega_0^{(s_i)}=s_i-1 & \text{ if } f_i = g_i \\
\end{cases} \, .
\]
It follows
\[
c_{0}^2 A_m(\fraction) = \frac{1}{\#\design^2}  \sum_{f \in \fraction}\sum_{g \in \fraction} (S_1^{(f,g)} \cdot \ldots \cdot S_m^{(f,g)})
\]
where
\[
S_i^{(f,g)}=
\begin{cases}
-1 & \text{ if } f_i \neq g_i \\
s_i-1 & \text{ if } f_i = g_i \\
\end{cases}, \; \quad i=1,\ldots,m.
\]
For $j<m$ it is sufficient to apply the formula above for all the subsets of size $j$ of the set $\{1,\ldots,m\}$
\[
c_{0}^2 A_j(\fraction) = \frac{1}{\#\design^2}  \sum_{f \in \fraction}\sum_{g \in \fraction} \sum_{\substack{A=\{a_1,\ldots,a_j\} \\ A \subseteq \{1,\ldots,m\}}} (S_{a_1}^{(f,g)} \cdot \ldots \cdot S_{a_j}^{(f,g)}).
\]
To complete the proof it is enough to observe that $c_0=\frac{n}{\#\design}$.
\end{proof}

As a consequence, the GWLP of a singleton is a constant depending only on $s_1, \ldots, s_m$, as stated below.

\begin{corollary} \label{GWLPsingleton}
Given a design point $f$ in a factorial design $\design=\design_1 \times \cdots \times \design_m$, the $j$-th element of the GWLP of $f$ is
\begin{equation} \label{singleton}
A_j(f) = \sum_{\substack{A=\{a_1,\ldots,a_j\} \\ A \subseteq \{1,\ldots,m\}}} (s_{a_1}-1) \cdots (s_{a_j}-1) \, ,
\end{equation}
where $s_1, \ldots, s_m$ are the number of levels of $\design_1, \ldots, \design_m$ respectively.
\end{corollary}

The formula in Eq.~\eqref{singleton} becomes very simple for symmetric designs. Indeed, when $s_1= \ldots = s_m=s$, we have
\begin{equation} \label{singleton-2}
A_j(f) = \binom{m}{j} (s-1)^{j} \, .
\end{equation}

Now we show some formulae to decompose the GWLP of a fraction. The proof of the first result can be found in \cite{fontana|rapallo:postSIS}. The subsequent results exploit the results in Prop.~\ref{pr:gwlp_mat} and Cor.~\ref{GWLPsingleton} and will be useful for our purpose, i.e., to choose the best design points to be removed from a given fraction.

In the following proposition we consider the union of $k$ fractions, $k\geq 2$. Let us consider fractions ${\mathcal F}_1, \ldots, {\mathcal F}_k$ with $n_1, \ldots, n_k$ design points, respectively. Let us denote by $R^{(i)}=\sum_{\alpha \in L} c_\alpha^{(i)}X^\alpha$ the counting function of $\fraction_i$, $i=1,\ldots,k$. When we consider the union ${\mathcal F} = {\mathcal F}_1 \cup \cdots \cup {\mathcal F}_k$ of size $n=n_1+ \ldots + n_k$, the counting function of $\mathcal F$ is clearly $R=\sum_{i=1}^k R^{(i)}$.

\begin{proposition} \label{union:prop:1}
With the notation above, let us consider fractions ${\mathcal F}_1, \ldots, {\mathcal F}_k$ with $n_1, \ldots, n_k$ design points, respectively, and their union ${\mathcal F} = {\mathcal F}_1 \cup \cdots \cup {\mathcal F}_k$. The $j$-th element of the GWLP of $\fraction$ is
\begin{equation} \label{GWLP-1}
A_j (\fraction) = \sum_{i=1}^k \frac {n_i^2} {n^2} A_j(\fraction_i) + \frac {(\#\design)^2}{n^2} \sum_{i_1 \ne i_2}  \sum_{|\alpha|_0=j}  c_\alpha^{(i_1)} \overline{c}_\alpha^{(i_2)}, \ j=0,\ldots,m \, .
\end{equation}
\end{proposition}
%

\begin{remark}
The term $\sum_{|\alpha|_0=j} c_\alpha^{(i_1)} c_\alpha^{(i_2)}$ in Eq.~\ref{GWLP-1} can be viewed as a kind of covariance between the coefficients of order $j$ of the two counting functions $R^{(i_1)}$ and $R^{(i_2)}$.
\end{remark}

We consider now two special cases of Prop.~\ref{union:prop:1}. In the first proposition we decompose a fraction as the union of singletons, while in the second one we explicitly write the formula for the GWLP of a fraction with one removed run.

\begin{proposition} \label{prop:union:2}
Let ${\mathcal F}=\{f_1, \ldots, f_n\}$ be a fraction  with $n$ runs, and let $R^{(f_i)}=\sum_\alpha c^{(f_i)}_\alpha X^\alpha$ be the indicator function of the point $f_i$. Then
\begin{equation} \label{unionsigletons}
A_j (\fraction) = \frac {1} {n^2}  \sum_{i=1}^n A_j(f_i) + \frac {(\#\design)^2}{n^2} \sum_{i_1 \ne i_2}  \sum_{|\alpha|_0=j}  c^{(f_{i_1})}_\alpha \overline{c}^{(f_{i_2})}_\alpha \; \; j=0,\ldots,m \, .
\end{equation}
\end{proposition}

Now, let us take a fraction $\mathcal F$ and a design point $f \in \mathcal F$. We denote with ${\mathcal F}_f$ the fraction with $n-1$ runs obtained by removing $f$ from $\mathcal F$.

\begin{proposition} \label{prop:union:3}
Let ${\mathcal F}=\{f_1, \ldots, f_n\}$ be a fraction  with $n$ runs, and let $R^{(f_i)}=\sum_\alpha c^{(f_i)}_\alpha X^\alpha$ be the indicator function of the point $f_i$. Then
\begin{equation} \label{removeone}
A_j (\fraction) = \left(\frac {n-1} n \right)^2 A_j({\mathcal F}_f) + \frac 1 {n^2} A_j(f) + \frac {(\#\design)^2}{n^2} \sum_{g \in {\mathcal F}, g \ne f}  \sum_{|\alpha|_0=j}  c^{(f)}_\alpha \overline{c}^{(g)}_\alpha
\end{equation}
for $j=0,\ldots,m$.
\end{proposition}

Notice that in the Equations \eqref{unionsigletons} and \eqref{removeone} the term involving a singleton is constant.

Before the use of the previous results in actual computations, some remarks are in order.

\begin{remark}
The last summand in Equations \eqref{GWLP-1}, \eqref{unionsigletons}, and \eqref{removeone} is independent on the choice of orthonormal contrasts, because it is the difference of GWLPs, see \cite{xu:wu}, page 1069. Therefore, the choice of the complex coding is due merely to computational reasons.
\end{remark}

\begin{remark}
From Eq.~\eqref{singleton}, we note that the elements $A_j(\fraction)$ of the GWLP of a fraction $\fraction$ depends only on the mutual position of the runs, and in this sense the elements $A_j(\fraction)$ of the GWLP of a fraction with two runs can be viewed as a kind of distance of the design points.
\end{remark}

\begin{remark}
Another consequence of Equations \eqref{singleton} and \eqref{unionsigletons} is that the terms 
\[
\sum_{|\alpha|_0=j}  c^{(f_{i_1})}_\alpha \overline{c}^{(f_{i_2})}_\alpha
\]
in the last summand of Eq.~\eqref{unionsigletons} can be used also for all sub-fractions.
\end{remark}

The results above suggest to introduce a sequence of matrices based on Eq.~\eqref{formula_S} to easily compute the GWLP of a fraction and its sub-fractions.

\begin{definition}
Given a fraction ${\mathcal F}$ with $n$ runs of a full factorial design with $m$ factors, define for each $j=1, \ldots, m$ the $n \times n$ matrix $W_j$ with generic element
\begin{equation}
W_j(f_{1},f_{2}) = {\#\design^2} \sum_{|\alpha|_0=j}  c^{(f_{1})}_\alpha \overline{c}^{(f_{2})}_\alpha = \sum_{\substack{A=\{a_1,\ldots,a_j\} \\ A \subseteq \{1,\ldots,m\}}} (S_{a_1}^{(f_1,f_2)} \cdot \ldots \cdot S_{a_j}^{(f_1,f_2)}) \, .
\end{equation}
\end{definition}
It follows from Proposition \ref{pr:gwlp_mat} that the j-th aberration can be written as
\[
A_j(\fraction)=\frac{1}{n^2} \sum_{f \in \fraction} \sum_{g \in \fraction} W_j(f,g), \;\;\; j=1,\ldots,m.
\]
It also follows that given a point $f \in \fraction$
\begin{equation}
n^2 A_j(\fraction) = (n-1)^2 A_j({\fraction}_f) + w_{j,f}
\label{eq:wj}
\end{equation}
where
\[
w_{j,f} = \sum_{c=1}^n W_j(f,c) + \sum_{r=1}^n W_j(r,f) - W_j(f,f) .
\]

Exploiting the formula in Eq.~\eqref{eq:wj} it is easy to use $w_{j,f}$ in order to choose the best point to be removed, i.e., the run which is candidate to be the last run of the OA. In the next section several examples are illustrated to show how this procedure works. Notice that in principle it would be easy to define a step-by-step algorithm removing one run at a time. In fact, the new matrices $W_j$ for the sub-fraction with $n-1$ runs can be obtained simply by deleting the row and column pertaining to the removed run, and therefore such matrices allows us to compute the GWLPs of fractions with two removed runs, and so on. However, as discussed in Sect.~\ref{not:hier}, such a procedure is in general not hierarchical.

\section{Examples} \label{sec:examples}

In this section we study the effect on the GWLP of the removal of one, two or three points from an OA of strength $t$. We consider both symmetric and mixed level OAs and we do not restrict the number of levels to be prime or prime power. Most of the examples here are chosen from the OA catalogue in \cite{eendebak:sito}.

\subsection{$OA(12,2^5,t=2)$} \label{not:hier}
We consider an OA with 12 runs, five 2-level factors and strength 2.
Writing the runs as columns and the factors as rows, the fraction $\fraction$ is
\[
\fraction =
\bordermatrix{ & f_1 & f_2 & f_3 & f_4 & f_5 & f_6 & f_7 & f_8 & f_9 & f_{10} & f_{11} & f_{12} \cr
                & 1 & 1 & 1 & 1 & 1 & 1 & -1 & -1 & -1 & -1 & -1 & -1 \cr
                & 1 & 1 & 1 & -1 & -1 & -1 & 1 & 1 & 1 & -1 & -1 & -1 \cr
                & 1 & 1 & -1 & 1 & -1 & -1 & 1 & -1 & -1 & 1 & 1 & -1 \cr
                & 1 & 1 & -1 & -1 & 1 & -1 & -1 & 1 & -1 & 1 & -1 & 1 \cr
                & 1 & -1 & 1 & 1 & -1 & -1 & -1 & 1 & -1 & -1 & 1 & 1}
\]
The GWLP of $\fraction$ is $A_{\fraction}=\left(1,0,0,A_{3}(\fraction)=1.111,A_{4}(\fraction)=0.5556,A_{5}(\fraction)=0\right)$.

We remove each of the twelve points from $\fraction$ and we compute the corresponding GWLPs. The results are reported in Table \ref{tab:ab_2_5_12}. We observe that, according to the results of Sect.~\ref{sec:formulae_GWLP}, $A_1(\fraction_f)=5/(12-1)^2=0.041$ and $A_2(\fraction_f)=10/(12-1)^2=0.083$. It is worth noting that there are two different GWLPs. More specifically there are $10$ fractions $\fraction_f$ with $A_3(\fraction_f)=1.140$ and $2$ fractions $\fraction_f$ with $A_3(\fraction_f)=1.405$.

\begin{table}
\caption{GWLPs of the fractions with one removed run for the OA in Sect.~\ref{not:hier}}
\label{tab:ab_2_5_12}       
\begin{tabular}{lrrrrr}
\hline\noalign{\smallskip}
point$^a$ & $A_1(\fraction_f)$ &	$A_2(\fraction_f)$ &	$A_3(\fraction_f)$ &	$A_4(\fraction_f)$ &	$A_5(\fraction_f)$ \\
\noalign{\smallskip}\hline\noalign{\smallskip}
$f_{1}$ &0.041 &	0.083 &	1.14 &	0.636 &	0.008 \\
$f_{2}$ &0.041 &	0.083 &	1.14 &	0.636 &	0.008 \\
$f_{3}$ &0.041 &	0.083 &	1.405 &	0.372 &	0.008 \\
$f_{4}$ &0.041 &	0.083 &	1.14 &	0.636 &	0.008 \\
$f_{5}$ &0.041 &	0.083 &	1.14 &	0.636 &	0.008 \\
$f_{6}$ &0.041 &	0.083 &	1.14 &	0.636 &	0.008 \\
$f_{7}$ &0.041 &	0.083 &	1.14 &	0.636 &	0.008 \\
$f_{8}$ &0.041 &	0.083 &	1.14 &	0.636 &	0.008 \\
$f_{9}$ &0.041 &	0.083 &	1.14 &	0.636 &	0.008 \\
$f_{10}$ &0.041 &	0.083 &	1.405 &	0.372 &	0.008 \\
$f_{11}$ &0.041 &	0.083 &	1.14 &	0.636 &	0.008 \\
$f_{12}$ & 0.041 &	0.083 &	1.14 &	0.636 &	0.008 \\
\noalign{\smallskip}\hline
\end{tabular}

$^a$this column specifies the removed run.
\end{table}

%

The symmetric matrix $W_3$ (as defined in Sect.~\ref{sec:formulae_GWLP}) is written in the columns labeled $f_1,\ldots,f_{12}$ of Table \ref{tab:ab_2_5_12_w_a3}. The last column of Table \ref{tab:ab_2_5_12_w_a3} reports the value of $w_{3,f}, f \in \fraction$.
It follows from Eq.~\eqref{eq:wj} that if we want to choose a single point $f$ to be removed in a way that $A_j(\fraction_f)$ is as small as possible we must select one of the points for which $w_{3,f}$ is as large as possible. In this case, for minimizing $A_3(\fraction_f)$ we must select $f \notin \{f_3,f_{10}\}$. These results are confirmed by the values of $A_3(\fraction_f)$ in Table \ref{tab:ab_2_5_12}.

A simple \emph{sequential} strategy can be defined. Once a run has been removed, the new $W_j$ matrix is obtained by simply deleting the row and the column corresponding to the removed point. Then the second point to be removed could be chosen by computing the new value of $w_{j,f}$ based on the new $W_j$ matrix. The problem is that, in general, this strategy does not lead to an optimal selection of the pair of points to be removed. In the case under study it is possible to verify that if we remove $f_{1}$ in the first step than the best possible choices for the second point to be removed would be $f_6$ or $f_9$ for which the aberration $A_1(\fraction_{f_1,f_6})= A_1(\fraction_{f_1,f_9})=0.04$. But if we select as the pair of points to be removed $\{f_3,f_{10}\}$ we obtain $A_1(\fraction_{f_3,f_{10}})=0$.
It follows that to have an optimal strategy the number of points to be removed must be fixed in advance.

\begin{table}
\caption{The $W_3$ matrix of $OA(12,2^5,t=2)$ for the OA in Sect.~\ref{not:hier}}
\label{tab:ab_2_5_12_w_a3}       
\begin{tabular}{rrrrrrrrrrrr|r}
\hline\noalign{\smallskip}
$f_1$ &	$f_2$ &	$f_3$ &	$f_4$ &	$f_5$ &	$f_6$ &	$f_7$ &	$f_8$ &	$f_9$ &	$f_{10}$ &	$f_{11}$ &	$f_{12}$ &	$w_{3,f}$ \\
\noalign{\smallskip}\hline\noalign{\smallskip}
10 &	-2 &	-2 &	-2 &	2 &	2 &	2 &	-2 &	2 &	2 &	2 &	2 &	22 \\
-2 &	10 &	2 &	2 &	-2 &	2 &	-2 &	2 &	2 &	-2 &	2 &	2 &	22 \\
-2 &	2 &	10 &	-2 &	2 &	-2 &	2 &	-2 &	-2 &	-10 &	2 &	2 &	-10 \\
-2 &	2 &	-2 &	10 &	2 &	-2 &	2 &	2 &	2 &	2 &	-2 &	2 &	22 \\
2 &	-2 &	2 &	2 &	10 &	-2 &	2 &	2 &	2 &	-2 &	2 &	-2 &	22 \\
2 &	2 &	-2 &	-2 &	-2 &	10 &	2 &	2 &	-2 &	2 &	2 &	2 &	22 \\
2 &	-2 &	2 &	2 &	2 &	2 &	10 &	2 &	-2 &	-2 &	-2 &	2 &	22 \\
-2 &	2 &	-2 &	2 &	2 &	2 &	2 &	10 &	-2 &	2 &	2 &	-2 &	22 \\
2 &	2 &	-2 &	2 &	2 &	-2 &	-2 &	-2 &	10 &	2 &	2 &	2 &	22 \\
2 &	-2 &	-10 &	2 &	-2 &	2 &	-2 &	2 &	2 &	10 &	-2 &	-2 &	-10 \\
2 &	2 &	2 &	-2 &	2 &	2 &	-2 &	2 &	2 &	-2 &	10 &	-2 &	22 \\
2 &	2 &	2 &	2 &	-2 &	2 &	2 &	-2 &	2 &	-2 &	-2 &	10 &	22 \\
\noalign{\smallskip}\hline
\end{tabular}
\end{table}

\subsection{Plackett-Burman $OA(12,2^{11},t=2)$} \label{pb:des}

We consider the Plackett-Burman design with 12 runs and eleven two-level factors. It has strength 2, i.e., resolution III.
We point out that the removal of even a single point leads to a design where the number of parameters to be estimated is larger than the number of runs.
We observe that any choice of one run to be removed leads to the same GWLP, which is reported in the first row of Table \ref{tab:ab_2_11_12_unique1}. Similarly any choice of two (three) points leads to the same GWLP which is reported in the second (third) row of Table \ref{tab:ab_2_11_12_unique1}.


\begin{table}
\caption{One to three points removed - unique GWLPs for the Plackett-Burman design in Sect.~\ref{pb:des}}
\label{tab:ab_2_11_12_unique1}       
\begin{tabular}{rrrrrrrr}
\hline\noalign{\smallskip}
$p$& $N$ &	$A_1$ &	$A_2$ &	$A_3$ & $\ldots$ &	$A_{10}$ &	$A_{11}$ \\
\noalign{\smallskip}\hline\noalign{\smallskip}
1& 12 &	0.091 &	0.455 &	19.545 &	$\ldots$ &	0.091 &	1 \\
2& 66 &	0.2 &	1 &	21 &	$\ldots$ & 0.2 &	1 \\
3 & 220 &	0.333 &	1.667 &	22.778  & $\ldots$ &	0.333 &	1 \\
\noalign{\smallskip}\hline
\end{tabular}

$p$ is the number of removed points, $N$ is the number of fractions with the same GWLP.
\end{table}

\subsection{$OA(18,2^1 3^3,t=2)$} \label{ex:symmetric}

We consider an orthogonal array with 18 runs, one two-level and 3 three-level factors, and strength 2. In this case any choice of one point to be removed leads to the same GWLP, which is reported in the first row of Table \ref{tab:ab_2333_18_unique}.
But different choices of pairs of points give different GWLPs. From Table \ref{tab:ab_2333_18_unique} for example we observe that $A_1$ varies between $0.023$ and $0.07$.

\begin{table}
\caption{One to two points removed - unique GWLPs for the OA in Sect.~\ref{ex:symmetric}}
\label{tab:ab_2333_18_unique}       
\begin{tabular}{rrrrrr}
\hline\noalign{\smallskip}
$p$ & $N$ &	$A_1$ &	$A_2$ &	$A_3$ &	$A_4$ \\
\noalign{\smallskip}\hline\noalign{\smallskip}
1 & 18 &	0.024 &	0.062 &	0.567 &	1.522 \\
2 & 27 &	0.023 &	0.188 &	0.617 &	1.547 \\
2 & 18 &	0.039 &	0.141 &	0.664 &	1.531 \\
2 & 27 &	0.047 &	0.117 &	0.688 &	1.523 \\
2 & 54 &	0.063 &	0.117 &	0.641 &	1.555 \\
2 & 27 &	0.07 &	0.117 &	0.617 &	1.57 \\
\noalign{\smallskip}\hline
\end{tabular}

$p$ is the number of removed points, $N$ is the number of fractions with the same GWLP.
\end{table}

\subsection{$OA(16,2^4 4^2,t=2)$} \label{ex:mixed}

As the last example we consider an orthogonal array with 16 runs, four 2-level and two 4-level factors of strength 2. The results are very similar to those of the previous case. The choice of one single point to be removed does not affect the GWLP while the choice of different pairs of points can lead to different GWLPs as reported in Table \ref{tab:ab_442222_16_unique}.

\begin{table}
\caption{One to two points removed - unique GWLPs for the OA in Sect.~\ref{ex:mixed}}
\label{tab:ab_442222_16_unique}       
\begin{tabular}{rrrrrrrr}
\hline\noalign{\smallskip}
$p$ & $N$ &	$A_1$ &	$A_2$ &	$A_3$ &	$A_4$ &	$A_5$ &	$A_6$ \\
\noalign{\smallskip}\hline\noalign{\smallskip}
1 & 16 &	0.044 &	0.173 &	6.311 &	8.316 &	0.187 &	1.036 \\
2 & 8 &	0.041 &	0.551 &	6.449 &	8.796 &	0.367 &	1.082 \\
2 &32 &	0.082 &	0.388 &	6.694 &	8.633 &	0.408 &	1.082 \\
2 &32 &	0.102 &	0.327 &	6.735 &	8.673 &	0.347 &	1.102 \\
2 &32 &	0.102 &	0.367 &	6.653 &	8.673 &	0.429 &	1.061 \\
2 &16 &	0.122 &	0.347 &	6.612 &	8.714 &	0.449 &	1.041 \\
\noalign{\smallskip}\hline
\end{tabular}

$p$ is the number of removed points, $N$ is the number of fractions with the same GWLP.
\end{table}

\section{Two-level designs} \label{sec:twolev}

In the case of two-level designs the full factorial design is ${\mathcal D}=\{-1,1\}^m$, the coefficients of the counting function are real numbers and therefore the aberrations in Eq.~\eqref{aberration} are simply
\[
a_\alpha = \left(  \frac{ c_{\alpha}}{c_{0}} \right)^2
\]
and all the computations yielding the elements of the GWLP involve only real numbers. Also the GWLP of a singleton assumes an easy form, since Eq.~\eqref{singleton-2} reduces to
\[
A_j(f) = \binom{m}{j} \, .
\]

But the main feature of the two-level case is that we can establish a recursive formula for the matrices $W_j$.

\begin{proposition}
Let $\mathcal F$ be a fraction with $n$ runs of a two-level full factorial design $\mathcal D=\{-1,1\}^m$, and denote by $X$ the design matrix of $\mathcal F$. The sequence of matrices $W_0,W_1, \ldots, W_m$ satisfy the recursive formula
\begin{eqnarray}
W_0 & = & J \\
W_1 & = & XX^t \\
W_j & = & \frac 1 {j!} \left( W_1\star W_{j-1} - (j-1)(m-j+2)W_{j-2} \right) \qquad j=2, \ldots , m
\end{eqnarray}
where $J$ is a $n \times n$ matrix with all entries equal to $1$ and $\star$ denotes the element-wise product of two matrices.
\end{proposition}

\begin{proof}
For $W_0$ the computation is trivial and for $W_1$ it is enough to observe that
\[
W_1(f_1,f_2) = \sum_{i} \omega^{(f_1)}_i \omega^{(f_2)}_i \, .
\]

To shorten the notation in the proof of the recursive formula, we write $\varphi_{i}=\omega_i^{(f_1)}\omega_i^{(f_2)}$ and $\varphi_{i_1 \cdots i_j}=\varphi_{i_1}\cdots\varphi_{i_j}$. We have:
\[
W_j(f_1,f_2) = \#{\mathcal D}^2 \sum_{|\alpha|_0=j} c^{(f_1)}_\alpha c^{(f_2)}_\alpha = \frac 1 {j!} \sum_{\substack{{i_1, \ldots, i_j} \\ {\mathrm{distinct}} }} \varphi_{i_1 \cdots i_j} =
\]
\[
= \frac 1 {j!} \sum_{i_1} \varphi_{i_1} \left( \sum_{\substack{i_2, \ldots, i_j \\ \mathrm{distinct} }} \varphi_{i_2 \cdots i_j}  - \sum_{\substack{i_2, \ldots, i_j \\ \mathrm{distinct} \\ i_1=i_2 }} \varphi_{i_1i_3 \cdots i_j} - \ldots - \sum_{\substack{i_2, \ldots, i_j \\ \mathrm{distinct} \\ i_1=i_j }} \varphi_{i_1i_2 \cdots i_{j-1}}   \right) =
\]
The first sum in the round bracket is $W_{j-1}(f_1,f_2)$ while the last $(j-1)$ terms are clearly equal. Thus,
\[
= \frac 1 {j!} \left\{ W_1(f_1,f_2) W_{j-1}(f_1,f_2) - (j-1) \sum_{i_1} \left( \varphi_{i_1}^2 \sum_{\substack{i_3, \ldots, i_j \\ \mathrm{distinct}} } \varphi_{i_3 \cdots i_j} \right) \right\} =
\]
changing the order of the sums in the last term
\[
= \frac 1 {j!} \left\{ W_1(f_1,f_2) W_{j-1}(f_1,f_2) - (j-1) \sum_{\substack{i_3, \ldots, i_j \\ \mathrm{distinct} }} \varphi_{i_3 \cdots i_j} \left( \sum_{i_1} \varphi_{i_1}^2 \right) \right\} =
\]
since $\varphi_{i_1}^2 = 1$ and $(m-j+2)$ terms are summed up we obtain
\[
= \frac 1 {j!} \left( W_1(f_1,f_2) W_{j-1}(f_1,f_2) - (j-1)(m-j+2) W_{j-2}(f_1,f_2) \right) \, .
\]
\end{proof}

\section{Final remarks} \label{sec:finalrem}

The formulae introduced in this paper allow us to easily compute the GWLP of Orthogonal Arrays with removed runs. However, as discussed in Sect.~\ref{not:hier}, the choice of the best GWLP is not hierarchical when more than one point is removed. Future work on this topic could focus on the characterization of some classes of Orthogonal Arrays with constant GWLP over all sub-fractions, when removing 1, 2, 3, and possibly more than 3 runs. It would also be interesting to compare the GWLP of such Orthogonal Arrays with the $R^2$ and the canonical correlations as in \cite{gromping|xu:14}, extending the analysis to designs which are different from Orthogonal Arrays. Finally, the connections between our approach based on the aberrations and the $D$-optimality criterion for some classical statistical models are worth exploring. Some efficient algorithms for finding such Orthogonal Arrays would be helpful, for example by exploiting the notion of mean aberration for mixed level Orthogonal Arrays introduced in \cite{fontana:rapallo:rogantin}.


\begin{thebibliography}{10}

\bibitem{butler07}
Neil~A. Butler and Victorino~M. Ramos.
\newblock Optimal additions to and deletions from two-level orthogonal arrays.
\newblock {\em J. R. Stat. Soc. Ser. B. Stat. Methodol.}, 69(1):51--61, 2007.

\bibitem{Chatzopoulos2011}
Stavros~A. Chatzopoulos, Fotini Kolyva-Machera, and Kashinath Chatterjee.
\newblock Optimality results on orthogonal arrays plus $p$ runs for $s^m$
  factorial experiments.
\newblock {\em Metrika}, 73(3):385--394, 2011.

\bibitem{cheng2004geometric}
Shao-Wei Cheng and Kenny~Q Ye.
\newblock Geometric isomorphism and minimum aberration for factorial designs
  with quantitative factors.
\newblock {\em Ann. Statist.}, 32(5):2168--2185, 2004.

\bibitem{dey2009fractional}
Aloke Dey and Rahul Mukerjee.
\newblock {\em Fractional Factorial Plans}.
\newblock John Wiley \& Sons, New York, 2009.

\bibitem{eendebak:sito}
Pieter Eendebak and Eric Schoen.
\newblock Complete series of non-isomorphic orthogonal arrays.
\newblock \texttt{http://pietereendebak.nl/oapage/}, 2018.
\newblock Accessed: 2018-07-31.

\bibitem{fontana|rapallo:postSIS}
Roberto Fontana and Fabio Rapallo.
\newblock Unions of orthogonal arrays and their aberrations via {H}ilbert
  bases.
\newblock Technical Report arXiv:1801.00591, 2018.
\newblock Submitted.

\bibitem{fontana:rapallo:rogantin}
Roberto Fontana, Fabio Rapallo, and Maria~Piera Rogantin.
\newblock Aberration in qualitative multilevel designs.
\newblock {\em J. Statist. Plann. Inference}, 174:1--10, 2016.

\bibitem{fries|hunter:80}
Arthur Fries and William~G. Hunter.
\newblock Minimum aberration $2^{k-p}$ designs.
\newblock {\em Technometrics}, 22(4):601--608, 1980.

\bibitem{gromping|xu:14}
Ulrike Gr\"omping and Hongquan Xu.
\newblock Generalized resolution for orthogonal arrays.
\newblock {\em Ann. Statist.}, 42(3):918--939, 2014.

\bibitem{hedayat2012orthogonal}
A.~S. Hedayat, Neil J.~A. Sloane, and John Stufken.
\newblock {\em Orthogonal Arrays: Theory and Applications}.
\newblock Springer, New York, 2012.

\bibitem{mukerjee2007modern}
Rahul Mukerjee and C.~F.~Jeff Wu.
\newblock {\em A Modern Theory of Factorial Design}.
\newblock Springer, New York, 2007.

\bibitem{pistone|rogantin:08}
Giovanni Pistone and Maria-Piera Rogantin.
\newblock Indicator function and complex coding for mixed fractional factorial
  designs.
\newblock {\em J. Statist. Plann. Inference}, 138(3):787--802, 2008.

\bibitem{street}
Deborah~J. Street and Emily~M. Bird.
\newblock ${D}$-optimal orthogonal array minus $t$ run designs.
\newblock {\em J. Stat. Theory Practice}, 12(3):575--594, 2018.

\bibitem{thestatistician}
P.C. Wang and H.W. Jan.
\newblock Designing two-level factorial experiments using {O}rthogonal {A}rrays
  when the run order is important.
\newblock {\em The Statistician}, 44(2):379--388, 1995.

\bibitem{xampeny}
Rafel Xampeny, Pere Grima, and Xavier Tort-Martorell.
\newblock Which runs to skip in two-level factorial designs when not all can be
  performed.
\newblock {\em Quality Engineering}, 2018.

\bibitem{xu:wu}
Hongquan Xu and C.~F.~Jeff Wu.
\newblock Generalized minimum aberration for asymmetrical fractional factorial
  designs.
\newblock {\em Ann. Statist.}, 29(4):1066--1077, 2001.

\end{thebibliography}

\end{document}